 \documentclass[final,leqno]{siamltex704}
\usepackage{amsmath}
\usepackage{amssymb}
\usepackage{graphicx}
\usepackage[notcite,notref]{showkeys}
\usepackage{tikz}
 \numberwithin{dummy}{section}

\newtheorem{algorithm}{Weak Galerkin Algorithm}

\newcommand{\bu}{{\bf u}}
\newcommand{\bq}{{\bf q}}
\newcommand{\bw}{{\bf w}}
\newcommand{\bx}{{\bf x}}

\newcommand{\be}{{\bf e}}
\newcommand{\bv}{{\bf v}}
\newcommand{\bH}{{\textbf{\textit{H}}}}
\newcommand{\bpsi}{{\boldsymbol\psi}}
\newcommand{\bepsilon}{{\boldsymbol\epsilon}}

\def\Q{{\mathbb Q}}
\def\T{{\mathcal T}}
\def\E{{\mathcal E}}

\def\pT{{\partial T}}
\def\l{{\langle}}
\def\r{{\rangle}}

\def\T{{\mathcal T}}
\def\E{{\mathcal E}}

\def\bbf{{\bf f}}
\def\bg{{\bf g}}
\def\bn{{\bf n}}
\def\bq{{\bf q}}

\def\3bar{{|\hspace{-.02in}|\hspace{-.02in}|}}

  \def\b#1{\mathbf{#1}} 
\def\a#1{\begin{align*}#1\end{align*}} \def\an#1{\begin{align}#1\end{align}} 
 
\def\p#1{\begin{pmatrix}#1\end{pmatrix}}

\title{A stabilizer-free pressure-robust finite element method for the Stokes equations }
\author{Xiu Ye\thanks{Department of
Mathematics, University of Arkansas at Little Rock, Little Rock, AR
72204 (xxye@ualr.edu). This research was supported in part by
National Science Foundation Grant DMS-1620016.}
\and
Shangyou Zhang\thanks{Department of
Mathematical Sciences, University of Delaware, Newark, DE 19716 (szhang@udel.edu).}
}
\begin{document}
\maketitle

\begin{abstract}
In this paper, we introduce a new finite element method for solving the Stokes equations in the primary velocity-pressure formulation.
This method employs $H(div)$ finite elements to approximate velocity, which leads to two unique advantages: exact divergence free velocity field and pressure-robustness. In addition, this method has a simple formulation without any stabilizer or penalty term.  Optimal-order error estimates are
established for the corresponding numerical approximation in various
norms. Extensive numerical investigations are conducted to test accuracy and robustness of the method and to confirm the theory. The numerical examples cover low and high order approximations up to the degree four, and 2D and 3D cases.
\end{abstract}


\pagestyle{myheadings}

\section{Introduction}
In this paper, we solve the Stokes problem which
seeks unknown functions $\bu$ and $p$ satisfying
\begin{eqnarray}
-\mu\Delta\bu+\nabla p&=& \bbf\quad
\mbox{in}\;\Omega,\label{moment}\\
\nabla\cdot\bu &=&0\quad \mbox{in}\;\Omega,\label{cont}\\
\bu&=&0\quad \mbox{on}\; \partial\Omega,\label{bc}
\end{eqnarray}
where $\mu$ denotes the fluid viscosity and $\Omega$ is a polygonal or polyhedral domain in
$\mathbb{R}^d\; (d=2,3)$.

The weak form in the primary velocity-pressure formulation for the
Stokes problem (\ref{moment})--(\ref{bc}) seeks $\bu\in
\bH_0^1(\Omega)$ and $p\in L_0^2(\Omega)$ satisfying
\begin{eqnarray}
(\mu\nabla\bu,\nabla\bv)-(\nabla\cdot\bv, p)&=&({\bf f}, \bv),\label{w1}\\
(\nabla\cdot\bu, q)&=&0,\label{w2}
\end{eqnarray}
for all $\bv\in \bH_0^1(\Omega)$ and $q\in L_0^2(\Omega)$.

The Stokes equations have many applications in fluid dynamics and been studied extensively by  researchers. For examples, finite element methods in the primary velocity-pressure formulation have been investigated in \cite{cr,gr} for continuous velocity approximations and in \cite{sst} for totally discontinuous velocity fields. In \cite{wwy,wyhdiv}, a $H(div)$ finite element method is proposed for the Stokes equations, i.e. the velocity is approximated by $H(div)$ finite element functions. There are two advantages of the method. First,  the numerical solution satisfies the divergence free condition exactly. Secondly, it is a pressure-robust discretization \cite{John}. Since the $H(div)$ finite element is discontinuous, a stabilizer with a penalty parameter is required in the formulation. The penalty parameters in \cite{wwy,wyhdiv} need to be large enough to ensure the well posedness of the problem.

The weak Galerkin (WG) finite element method is an effective and robust numerical technique for the approximate solution of partial differential equations, introduced in \cite{wy,wymix}. The WG methods use discontinuous piecewise polynomials as approximation on polytopal meshes.   The main idea of weak Galerkin finite
element methods is the use of weak functions and their corresponding weak derivatives in algorithm design. With the introduction of weak derivative, the WG method has a simple formulation and  no need to tune penalty parameters. The Stokes problems and the coupled Stoke-Darcy problems have been studied by the weak Galerkin methods  in \cite{cww,cfx,lllc,wy-stokes} and by the modified weak Galerkin methods in \cite{mwy,tzz}.

In this paper, we develop a new finite element method for solving the Stokes equations in (\ref{moment})-(\ref{bc}). Like the finite element method in \cite{wyhdiv}, we use $H(div)$ finite  element for velocity. As a result, our new finite element method has exact divergence free velocity field and is pressure-robust. Owe to the introduction of weak gradient in  \cite{mwy}, this new finite element method has the following simple formulation: seek $\bu_h\in V_h$ and $p_h\in W_h$ satisfying
\begin{eqnarray}
(\mu\nabla_w\bu_h,\nabla_w\bv)-(\nabla\cdot\bv, p_h)&=&({\bf f},\bv)\quad\forall \bv\in V_h ,\label{ww1}\\
(\nabla\cdot\bu_h, q)&=&0\quad\quad\forall q\in W_h,\label{ww2}
\end{eqnarray}
where $\nabla_w$ is the weak gradient which will be defined later. Unlike the $H(div)$ finite element method in \cite{wyhdiv}, there are no stabilizers nor penalty parameters in our new finite element formulations (\ref{ww1})-(\ref{ww2}).

The optimal order error estimates are established for the corresponding finite element approximations for velocity and pressure. Our theory and numerical tests demonstrate the pressure-robustness of the method.  Extensive numerical examples are tested for the finite elements with different degrees up to $P_4$ polynomials and for different dimensions, 2D and 3D.

\section{ Finite Element Method}
We use standard definitions for the Sobolev spaces $H^s(D)$ and their associated
inner products $(\cdot, \cdot)_{s,D}$, norms $\|\cdot\|_{s,D}$, and seminorms $|\cdot|_{s,D}$ for $s\ge 0$.
When $D = \Omega$, we drop the subscript $D$ in the norm and inner product notation.
 We also use $L^2_0(\Omega)$ to denote the subspace of $L^2(\Omega)$ consisting of functions with mean value zero.

Let ${\cal T}_h$ be a   shape regular partition of the domain $\Omega$
with mesh size $h$ that consists of triangles/tetrahedrons. Denote by ${\cal E}_h$ the set
of all  flat faces in ${\cal T}_h$, and let ${\cal
E}_h^0={\cal E}_h\backslash\partial\Omega$ be the set of all
interior faces.

For $k\ge 1$ and given $\T_h$, define two finite element spaces  for velocity
\begin{eqnarray}
V_h &=&\left\{ \bv\in H({\rm div},\Omega):\ \bv|_{T}\in [P_{k}(T)]^d,\;\forall T\in\T_h,\;\bv\cdot\bn|_{\partial\Omega}=0 \right\}.\label{vh}
\end{eqnarray}
and for pressure
\begin{equation}
W_h =\left\{q\in L_0^2(\Omega): \ q|_T\in P_{k-1}(T)\right\}.\label{wh}
\end{equation}

Let $T_1$ and $T_2$ be two triangles/tetrahedrons in $\T_h$
sharing $e\in\E_h$.  For $e\in\E_h$ and $\bv\in V_h+\bH_0^1(\Omega) $, the jump $[\bv]$ is defined as
\begin{equation}\label{jump}
[\bv]=\bv\quad {\rm if} \;e\subset \partial\Omega,\quad [\bv]=\bv|_{T_1}-\bv|_{T_2}\;\; {\rm if} \;e\in\E_h^0.
\end{equation}
The order of $T_1$ and $T_2$ is not essential.

For $e\in\E_h$ and $\bv\in V_h+\bH_0^1(\Omega)$, the average $\{v\}$ is defined  as
\begin{equation}\label{avg}
\{\bv\}={\bf 0}\quad {\rm if} \;e\subset \partial\Omega,\quad \{\bv\}=\frac12(\bv|_{T_1}+\bv|_{T_2})\;\; {\rm if} \;e\in\E_h^0.
\end{equation}

For a function $\bv\in V_h+\bH_0^1(\Omega)$, its weak gradient $\nabla_w\bv\in \prod_{T\in\T_h} [P_{k+1}(T)]^{d\times d}$ is defined on each $T\in T_h$ by
\begin{equation}\label{wg}
(\nabla_w\bv,\  \tau)_T = -(\bv,\  \nabla\cdot \tau)_T+
\l\{\bv\}, \ \tau\cdot\bn \r_\pT\quad\forall\tau\in [P_{k+1}(T)]^{d\times d}.
\end{equation} 
For simplicity, we adopt the following notations,
\begin{eqnarray*}
(v,w)_{\T_h} &=&\sum_{T\in\T_h}(v,w)_T=\sum_{T\in\T_h}\int_T vw d\bx,\\
 \l v,w\r_{\partial\T_h}&=&\sum_{T\in\T_h} \l v,w\r_\pT=\sum_{T\in\T_h} \int_\pT vw ds.
\end{eqnarray*}
Then we have the following simple  finite element scheme without stabilizers.

\smallskip

\begin{algorithm}
A numerical approximation for (\ref{moment})-(\ref{bc}) can be
obtained by seeking $\bu_h\in V_h$ and $p_h\in W_h$ such that for all $\bv\in V_h$ and $q\in W_h$,
\begin{eqnarray}
(\mu\nabla_w\bu_h,\ \nabla_w\bv)-(\nabla\cdot\bv,\;p_h)&=&(f,\;\bv),\label{wg1}\\
(\nabla\cdot\bu_h,\;q)&=&0.\label{wg2}
\end{eqnarray}
\end{algorithm}

Let $\Q_h$ be the element-wise defined $L^2$ projection onto the space $[P_{k+1}(T)]^{d\times d}$ for $T\in\T_h$.

\begin{lemma}
Let $\boldsymbol\phi\in \bH_0^1(\Omega)$, then on  $T\in\T_h$
\begin{eqnarray}
\nabla_w \boldsymbol\phi &=&\Q_h\nabla\boldsymbol\phi.\label{key}
\end{eqnarray}
\end{lemma}
\begin{proof}
Using (\ref{wg}) and  integration by parts, we have that for
any $\tau\in [P_{k+1}(T)]^{d\times d}$
\begin{eqnarray*}
(\nabla_w \boldsymbol\phi,\tau)_T &=& -(\boldsymbol\phi,\nabla\cdot\tau)_T
+\langle \{\boldsymbol\phi\},\tau\cdot\bn\rangle_{\pT}\\
&=& -(\boldsymbol\phi,\nabla\cdot\tau)_T
+\langle \boldsymbol\phi,\tau\cdot\bn\rangle_{\pT}\\
&=&(\nabla \boldsymbol\phi,\tau)_T=(\Q_h\nabla\phi,\tau)_T,
\end{eqnarray*}
which implies the desired identity (\ref{key}).
\end{proof}

For any function $\varphi\in H^1(T)$, the following trace
inequality holds true (see \cite{wymix} for details):
\begin{equation}\label{trace}
\|\varphi\|_{e}^2 \leq C \left( h_T^{-1} \|\varphi\|_T^2 + h_T
\|\nabla \varphi\|_{T}^2\right).
\end{equation}

\section{Well Posedness}
We assume  $\mu=1$ for simplicity. We start this section by introducing two semi-norms $\3bar \bv\3bar$ and  $\|\bv\|_{1,h}$
   for any $\bv\in V_h \cup \bH_0^1(\Omega)$ as follows:
\begin{eqnarray}
\3bar \bv\3bar^2 &=& \sum_{T\in\T_h}(\nabla_w\bv,\nabla_w\bv)_T, \label{norm1}\\
\|\bv\|_{1,h}^2&=&\sum_{T\in \T_h}\|\nabla \bv\|_T^2+\sum_{e\in\E_h}h_e^{-1}\|[\bv]\|_{e}^2.\label{norm2}
\end{eqnarray}

It is easy to see that $\|\bv\|_{1,h}$ defines a norm in $V_h$.
The following norm equivalence has been proved in  \cite{aw},
\begin{equation}\label{happy}
C_1\|\bv\|_{1,h}\le \3bar \bv\3bar\le C_2 \|\bv\|_{1,h} \quad\forall \bv\in V_h.
\end{equation}

The space $H({\rm div};\Omega)$ is defined as the set of vector-valued functions on $\Omega$ which,
together with their divergence, are square integrable; i.e.,
\[
H({\rm div}; \Omega)=\left\{ \bv\in [L^2(\Omega)]^d:\;
\nabla\cdot\bv \in L^2(\Omega)\right\}.
\]

Define a projection  $\Pi_h$ for $\tau\in H({\rm div},\Omega)$ (see \cite{bf}) such that
$\Pi_h\tau\in V_h$ and on each $T\in {\cal T}_h$
\begin{eqnarray}
(\nabla\cdot\tau,\;v)_T&=&(\nabla\cdot\Pi_h\tau,\;v)_T  \quad \forall v\in P_{k-1}(T).\label{key2}
\end{eqnarray}

The inf-sup condition for the finite element  formulation (\ref{wg1})-(\ref{wg2}) will be derived in the following lemma.

\smallskip
\begin{lemma}\label{Lemma:inf-sup}
There exists a positive constant $\beta$ independent of $h$ such that for all $\rho\in W_h$,
\begin{equation}\label{inf-sup}
\sup_{\bv\in V_h}\frac{(\nabla\cdot\bv,\rho)}{\3bar\bv\3bar}\ge \beta
\|\rho\|.
\end{equation}
\end{lemma}

\begin{proof}
For any given $\rho\in W_h\subset L_0^2(\Omega)$, it is known
\cite{gr} that there exists a function
$\tilde\bv\in \bH_0^1(\Omega)$ such that
\begin{equation}\label{c-inf-sup}
\frac{(\nabla\cdot\tilde\bv,\rho)}{\|\tilde\bv\|_1}\ge C\|\rho\|,
\end{equation}
where $C>0$ is a constant independent of $h$. By
setting $\bv=\Pi_h\tilde{\bv}\in V_h$, we prove next that the following
holds true
\begin{equation}\label{m9}
\3bar\bv\3bar\le C\|\tilde{\bv}\|_1.
\end{equation}
It follows from (\ref{happy}), (\ref{trace}) and $\tilde{\bv}\in \bH_0^1(\Omega)$,
\begin{eqnarray*}
\3bar \bv\3bar^2&\le& C\|\bv\|_{1,h}^2=C(\sum_{T\in \T_h}\|\nabla \bv\|_T^2+\sum_{e\in\E_h^0}h_e^{-1}\|[\bv]\|_{e}^2)\\
&\le&C \sum_{T\in \T_h}\|\nabla \Pi_h\tilde{\bv}\|_T^2+\sum_{e\in\E_h^0}h_e^{-1}\|[\Pi_h\tilde{\bv}-\tilde{\bv}]\|_{e}^2\\
&\le& C\|\tilde{\bv}\|_1^2,
\end{eqnarray*}
which implies the  inequality (\ref{m9}). It follows from (\ref{key2}) that
\begin{eqnarray*}
(\nabla\cdot\bv,\;\rho)&=&(\nabla\cdot\Pi_h\tilde\bv,\;\rho)=(\nabla\cdot\tilde\bv,\;\rho).
\end{eqnarray*}
Using the above equation, (\ref{c-inf-sup}) and (\ref{m9}), we have
\begin{eqnarray*}
\frac{|(\nabla\cdot\bv,\rho)|} {\3bar\bv\3bar} &\ge &
\frac{|(\nabla\cdot\tilde\bv,\rho)|}{C\|\tilde\bv\|_1}\ge
\beta\|\rho\|,
\end{eqnarray*}
for a positive constant $\beta$. This completes the proof of the
lemma.
\end{proof}

\begin{lemma}
The weak Galerkin method (\ref{wg1})-(\ref{wg2}) has a unique solution.
\end{lemma}

\smallskip

\begin{proof}
It suffices to show that zero is the only solution of
(\ref{wg1})-(\ref{wg2}) if $\bbf={\bf 0}$. To this end, let $\bbf={\bf 0}$ and
take $\bv=\bu_h$ in (\ref{wg1}) and $q=p_h$ in (\ref{wg2}). By adding the
two resulting equations, we obtain
\[
(\nabla_w\bu_h,\ \nabla_w\bu_h)=0,
\]
which implies that $\nabla_w \bu_h=0$ on each element $T$. By (\ref{happy}), we have $\|\bu_h\|_{1,h}=0$ which implies that $\bu_h=0$.

Since $\bu_h={\bf 0}$ and $\bbf={\bf 0}$, the equation (\ref{wg1}) becomes $(\nabla\cdot\bv,\ p_h)=0$ for any $\bv\in V_h$. Then the inf-sup condition (\ref{inf-sup}) implies $p_h=0$. We have proved the lemma.
\end{proof}

\section{Error Equations}
In this section, we will derive the equations that the errors satisfy. First we define an element-wise $L^2$ projection $Q_h$ onto the local space $P_{k-1}(T)$ for $T\in\T_h$.
Let $\be_h=\Pi_h\bu-\bu_h$, $\bepsilon_h=\bu-\bu_h$ and $\varepsilon_h=Q_hp-p_h$.

\begin{lemma}
For any $\bv\in V_h$ and $q\in W_h$, the following error equations hold true,
\begin{eqnarray}
(\nabla_w\be_h,\; \nabla_w\bv)-(\varepsilon_h,\;\nabla\cdot\bv)&=&\ell_1(\bu,\bv)-\ell_2(\bu,\bv),\label{ee1}\\
(\nabla\cdot\be_h,\ q)&=&0,\label{ee2}
\end{eqnarray}
where
\begin{eqnarray}
\ell_1(\bu,\ \bv)&=&\l\bv-\{\bv\},\ \nabla\bu\cdot\bn-\Q_h(\nabla\bu)\cdot\bn\r_{\partial\T_h},\label{l1}\\
\ell_2(\bu,\bv)&=&(\nabla_w (\bu-\Pi_h\bu),\nabla_w\bv).\label{l3}
\end{eqnarray} 
\end{lemma}

\begin{proof}
First, we test (\ref{moment}) by
$\bv\in V_h$ to obtain
\begin{equation}\label{mm0}
-(\Delta\bu,\;\bv)+(\nabla p,\ \bv)=(\bbf,\; \bv).
\end{equation}
Integration by parts gives
\begin{equation}\label{mm1}
-(\Delta\bu,\;\bv)=(\nabla \bu,\nabla \bv)_{\T_h}- \langle
\nabla \bu\cdot\bn,\bv-\{\bv\}\rangle_{\partial\T_h},
\end{equation}
where we use the fact $\langle\nabla \bu\cdot\bn,\{\bv\}\rangle_{\pT_h}=0$.
It follows from integration by parts, (\ref{wg}) and (\ref{key}),
\begin{eqnarray}
(\nabla \bu,\nabla \bv)_{\T_h}&=&(\Q_h\nabla  \bu,\nabla \bv)_{\T_h}\nonumber\\
&=&-(\bv,\nabla\cdot (\Q_h\nabla \bu))_{\T_h}+\langle \bv, \Q_h\nabla \bu\cdot\bn\rangle_{\partial\T_h}\nonumber\\
&=&(\Q_h\nabla \bu, \nabla_w \bv)_{\T_h}+\langle \bv-\{\bv\},\Q_h\nabla \bu\cdot\bn\rangle_{\partial\T_h}\nonumber\\
&=&( \nabla_w \bu, \nabla_w \bv)+\langle \bv-\{\bv\},\Q_h\nabla \bu\cdot\bn\rangle_{\partial\T_h}.\label{j1}
\end{eqnarray}
Combining (\ref{mm1}) and (\ref{j1}) gives
\begin{eqnarray}
-(\Delta\bu,\;\bv)&=&(\nabla_w \bu,\nabla_w\bv)-\ell_1(\bu,\bv).\label{mm2}
\end{eqnarray}
Using integration by parts and $\bv\in V_h$,  we have
\begin{equation}\label{mm3}
(\nabla p,\ \bv)= -(p,\nabla\cdot\bv)_{\T_h}+\l p, \bv\cdot\bn\r_{\partial\T_h}=-(p,\nabla\cdot\bv)_{\T_h}=-(Q_hp,\nabla\cdot\bv)_{\T_h}.
\end{equation}
Substituting (\ref{mm2}) and (\ref{mm3}) into (\ref{mm0}) gives
\begin{equation}\label{mm4}
(\nabla_w\bu,\nabla_w\bv)-(Q_hp,\nabla\cdot\bv)_{\T_h}=(\bbf,\bv)+\ell_1(\bu,\bv).
\end{equation}
The difference of (\ref{mm4}) and (\ref{wg1}) implies
\begin{equation}\label{mm10}
(\nabla_w\bepsilon_h,\nabla_w\bv)-(\varepsilon_h,\nabla\cdot\bv)_{\T_h}=\ell_1(\bu,\bv)\quad\forall\bv\in V_h.
\end{equation}
Adding and subtracting $(\nabla_w\Pi_h\bu,\nabla_w\bv)$ in (\ref{mm10}), we have
\begin{equation}\label{mm11}
(\nabla_w\be_h,\nabla_w\bv) -(\varepsilon_h,\nabla\cdot\bv)_{\T_h}=\ell_1(\bu,\bv)-\ell_2(\bu,\bv),
\end{equation}
which implies (\ref{ee1}).

Testing equation (\ref{cont}) by $q\in W_h$ and using (\ref{key2}) give
\begin{equation}\label{mm5}
(\nabla\cdot\bu,\ q)=(\nabla\cdot\Pi_h\bu,\ q)_{\T_h}=0.
\end{equation}
The difference of (\ref{mm5}) and (\ref{wg2}) implies (\ref{ee2}). We have proved the lemma.
\end{proof}

\section{Error Estimates in Energy Norm}\label{Section:error-analysis}
In this section, we shall establish optimal order error estimates
for the velocity approximation $\bu_h$ in $\3bar\cdot\3bar$ norm and for the pressure approximation $p_h$ in
the standard $L^2$ norm.

It is easy to see that the following equations hold true for $\{\bv\}$ defined in (\ref{avg}),
\begin{equation}\label{jp}
\|\bv-\{\bv\}\|_e=\|[\bv]\|_e\quad {\rm if} \;e\subset \partial\Omega,\quad\|\bv-\{\bv\}\|_e=\frac12\|[\bv]\|_e\;\; {\rm if} \;e\in\E_h^0.
\end{equation}

\begin{lemma}
Let $\bw\in \bH^{k+1}(\Omega)$ and $\rho\in H^k(\Omega)$ and
$\bv\in V_h$. Assume that the finite element partition $\T_h$ is
shape regular. Then, the following estimates hold true
\begin{eqnarray}
|\ell_1(\bw,\ \bv)|&\le& Ch^{k}|\bw|_{k+1}\3bar \bv\3bar,\label{mmm2}\\
|\ell_2(\bw,\ \bv)|&\le& Ch^{k}|\bw|_{k+1}\3bar \bv\3bar.\label{mmm3}
\end{eqnarray}
\end{lemma}

\begin{proof}
Using the Cauchy-Schwarz inequality, the trace inequality (\ref{trace}), (\ref{jp}) and  (\ref{happy}), we have
\begin{eqnarray*}
|\ell_1(\bw,\ \bv)|&=&|\sum_{T\in\T_h}\l\bv-\{\bv\},\ \nabla\bw\cdot\bn-\Q_h(\nabla\bw)\cdot\bn\r_\pT|\\
&\le & C \sum_{T\in\T_h}\|\nabla \bw-\Q_h\nabla \bw\|_{\pT}
\|\bv-\{\bv\}\|_\pT\nonumber\\
&\le & C \left(\sum_{T\in\T_h}h_T\|(\nabla \bw-\Q_h\nabla \bw)\|_{\pT}^2\right)^{\frac12}
\left(\sum_{e\in\E_h}h_e^{-1}\|[\bv]\|_e^2\right)^{\frac12}\\
&\le & Ch^{k}|\bw|_{k+1}\3bar \bv\3bar.
\end{eqnarray*}
Next we estimate $|\ell_2(\bw, \bv)|=|(\nabla_w (\bw-\Pi_h\bw),\nabla_w\bv)|$.
It follows from (\ref{wg}), integration  by parts, (\ref{trace}) and (\ref{jp}) that for any $\bq\in [P_{k+1}(T)]^{d\times d}$,
\begin{eqnarray}
 \nonumber &&|(\nabla_w(\bw-\Pi_h\bw), \bq)_{T}|\\
  &=&|-(\bw-\Pi_h\bw, \nabla\cdot\bq)_{T}+\l \bw-\{\Pi_h\bw\}, \bq\cdot\bn\r_{\pT}|\nonumber\\
&=&|(\nabla (\bw-\Pi_h\bw), \bq)_{T}+\l \Pi_h\bw-\{\Pi_h\bw\}, \bq\cdot\bn\r_{\pT}|\nonumber\\
&\le& \|\nabla (\bw-\Pi_h\bw)\|_T\|\bq\|_T+Ch^{-1/2}\|[\Pi_h\bw]\|_\pT\|\bq\|_T\nonumber\\
&=& \|\nabla (\bw-\Pi_h\bw)\|_T\|\bq\|_T+Ch^{-1/2}\|[\bw-\Pi_h\bw]\|_\pT\|\bq\|_T\nonumber\\
&\le& Ch^k|\bw|_{k+1, T}\|\bq\|_T.\label{m30}
\end{eqnarray}
Letting $\bq=\nabla_w\bv$ in the above equation and taking summation over $T$, we have
\begin{eqnarray*}
|\ell_2(\bw,\ \bv)|&\le& Ch^{k}|\bw|_{k+1}\3bar \bv\3bar.
\end{eqnarray*}
We have proved the lemma.
\end{proof}

\begin{theorem}\label{h1-bd}
Let $(\bu,p)\in  \bH_0^1(\Omega)\cap \bH^{k+1}(\Omega)\times
(L_0^2(\Omega)\cap H^{k}(\Omega))$ with $k\ge 1$ and $(\bu_h,p_h)\in
V_h\times W_h$ be the solution of (\ref{moment})-(\ref{bc}) and
(\ref{wg1})-(\ref{wg2}), respectively. Then, the following error
estimates hold  true
\begin{eqnarray}
\3bar  \bu-\bu_h\3bar &\le& Ch^{k}|\bu|_{k+1},\label{errv}\\
\|Q_hp-p_h\|&\le&Ch^{k}|\bu|_{k+1},\label{errp}\\
\|p-p_h\|&\le& Ch^{k}(|\bu|_{k+1}+|p|_{k}).\label{err-p}
\end{eqnarray}
\end{theorem}

\smallskip

\begin{proof}
By letting $\bv=\be_h$ in (\ref{ee1}) and $q=\varepsilon_h$ in
(\ref{ee2}) and adding the two resulting equations, we have
\begin{eqnarray}
\3bar \be_h\3bar^2&=&|\ell_1(\bu,\bv)-\ell_2(\bu,\bv)|.\label{main}
\end{eqnarray}
It then follows from (\ref{mmm2}) and (\ref{mmm3}) that
\begin{equation}\label{b-u}
\3bar \be_h\3bar^2 \le Ch^{k}|\bu|_{k+1}\3bar \be_h\3bar.
\end{equation}
By the triangle inequality and \eqref{b-u}, (\ref{errv}) holds.
To estimate
$\|\varepsilon_h\|$, we have from (\ref{ee1}) that
\[
(\varepsilon_h, \nabla\cdot\bv)=(\nabla_w\be_h,\nabla_w\bv)-\ell(\bu, \bv).
\]
Using the equation above (\ref{b-u}) and
(\ref{mmm2}), we arrive at
\[
|(\varepsilon_h, \nabla\cdot\bv)|\le
Ch^{k}|\bu|_{k+1}\3bar\bv\3bar.
\]
Combining the above estimate with the {\em inf-sup} condition
(\ref{inf-sup}) gives
\[
\|\varepsilon_h\|\le Ch^{k}|\bu|_{k+1},
\]
which yields the desired estimate (\ref{errp}).
 \eqref{err-p} follows by the triangle inequality.
\end{proof}

\section{Error Estimates in $L^2$ Norm}

In this section, we shall derive an $L^2$-error estimate
for the velocity approximation through a duality argument. Recall that $\be_h=\Pi_h\bu-\bu_h$ and $\bepsilon_h=\bu-\bu_h$. To this
end, consider the problem of seeking $(\bpsi,\xi)$ such that
\begin{eqnarray}
-\Delta\bpsi+\nabla \xi&=\bepsilon_h &\quad \mbox{in}\;\Omega,\label{dual-m}\\
\nabla\cdot\bpsi&=0 &\quad\mbox{in}\;\Omega,\label{dual-c}\\
\bpsi&= 0 &\quad\mbox{on}\;\partial\Omega.\label{dual-bc}
\end{eqnarray}
Assume that the dual problem has the $\bH^{2}(\Omega)\times
H^1(\Omega)$-regularity property in the sense that the solution
$(\bpsi,\xi)\in \bH^{2}(\Omega)\times H^1(\Omega)$ and the
following a priori estimate holds true:
\begin{equation}\label{reg}
\|\bpsi\|_{2}+\|\xi\|_1\le C\|\bepsilon_h\|.
\end{equation}

\medskip
\begin{theorem}
Let  $(\bu_h,p_h)\in V_h\times W_h$ be the solution of
(\ref{wg1})-(\ref{wg2}). Assume that (\ref{reg}) holds true.  Then, we have
\begin{equation}\label{l2err}
\|\bu-\bu_h\|\le Ch^{k+1}(|\bu|_{k+1}+|p|_{k}).
\end{equation}
\end{theorem}

\begin{proof}
Testing (\ref{dual-m}) by $\bepsilon_h$   gives
\begin{eqnarray}
(\bepsilon_h, \bepsilon_h)&=&-(\Delta\bpsi,\;\bepsilon_h)+(\nabla \xi,\ \bepsilon_h).\label{m20}
\end{eqnarray}

Using integration by parts and the fact $\l \nabla\bpsi\cdot\bn,\{\bepsilon_h\}\r_{\partial\T_h}=0$, then
\begin{eqnarray*}
-(\Delta\bpsi,\bepsilon_h)
&=&(\nabla \bpsi,\ \nabla\bepsilon_h)_{\T_h}-\l
\nabla\bpsi\cdot\bn,\ \bepsilon_h- \{\bepsilon_h\}\r_{\partial\T_h}\\
&=&(\Q_h\nabla \bpsi,\ \nabla\bepsilon_h)_{\T_h}+(\nabla\bpsi-\Q_h\nabla \bpsi,\ \nabla\bepsilon_h)_{\T_h}-\l
\nabla\bpsi\cdot\bn,\ \bepsilon_h- \{\bepsilon_h\}\r_{\partial\T_h}\\
&=&-(\nabla\cdot\Q_h\nabla \bpsi,\ \bepsilon_h)_{\T_h}+\l \Q_h\nabla\bpsi\cdot\bn,\ \bepsilon_h\r_{\partial\T_h}\\
&+&(\nabla\bpsi-\Q_h\nabla \bpsi,\ \nabla\bepsilon_h)_{\T_h}-\l\nabla\bpsi\cdot\bn,\ \bepsilon_h- \{\bepsilon_h\}\r_{\partial\T_h}\\
&=&(\Q_h\nabla \bpsi,\ \nabla_w\bepsilon_h)_{\T_h}+\l\Q_h\nabla\bpsi\cdot\bn,\ \bepsilon_h-\{\bepsilon_h\}\r_{\partial\T_h}\\
&+&(\nabla\bpsi-\Q_h\nabla \bpsi,\ \nabla\bepsilon_h)_{\T_h}-\l\nabla\bpsi\cdot\bn,\ \bepsilon_h- \{\bepsilon_h\}\r_{\partial\T_h}\\
&=&(\Q_h\nabla \bpsi,\ \nabla_w\bepsilon_h)_{\T_h}+(\nabla\bpsi-\Q_h\nabla \bpsi,\ \nabla\bepsilon_h)_{\T_h}-\ell_1(\bpsi,\bepsilon_h).
\end{eqnarray*}
It follows from (\ref{key}) that
\begin{eqnarray*}
(\Q_h\nabla \bpsi,\ \nabla_w\bepsilon_h)_{\T_h}&=&(\nabla_w \bpsi,\;\nabla_w\bepsilon_h)_{\T_h}\\
&=&(\nabla_w \Pi_h\bpsi,\;\nabla_w\bepsilon_h)_{\T_h}+(\nabla_w (\bpsi-\Pi_h\bpsi),\;\nabla_w\bepsilon_h)_{\T_h}.
\end{eqnarray*}
The two equations above imply that
\begin{eqnarray}
-(\Delta\bpsi,\bepsilon_h)&=&(\nabla_w \Pi_h\bpsi,\;\nabla_w\bepsilon_h)_{\T_h}+(\nabla_w (\bpsi-\Pi_h\bpsi),\;\nabla_w\bepsilon_h)_{\T_h}\nonumber\\
&+&(\nabla\bpsi-\Q_h\nabla \bpsi,\ \nabla\bepsilon_h)_{\T_h}-\ell_1(\bpsi,\bepsilon_h).\label{m40}
\end{eqnarray}
It follows from integration by parts and (\ref{cont}) and (\ref{wg2})
\begin{eqnarray}
(\nabla \xi,\ \bepsilon_h)&=&0.\label{m41}
\end{eqnarray}
The equation  (\ref{mm10}) and the fact $(\epsilon_h,\nabla\cdot\Pi_h\bpsi)_{\T_h}=0$ give
\begin{eqnarray}
(\nabla_w \Pi_h\bpsi,\;\nabla_w\bepsilon_h)_{\T_h}=\ell_1(\bu,\Pi\bpsi_h).\label{m42}
\end{eqnarray}
Combining (\ref{m20})-(\ref{m42}), we have
\begin{eqnarray}
\|\bepsilon_h\|^2&=&\ell_1(\bu,\Pi_h\bpsi)+(\nabla_w(\bpsi-\Pi_h\bpsi),
   \;\nabla_w\bepsilon_h)_{\T_h}\nonumber\\
&+&(\nabla\bpsi-\Q_h\nabla \bpsi,\ \nabla\bepsilon_h)_{\T_h}-\ell_1(\bpsi, \bepsilon_h).\label{mmm22}
\end{eqnarray}

Next we will estimate all the terms on the right hand side of (\ref{mmm22}). Using the Cauchy-Schwarz inequality, the trace inequality (\ref{trace}) and the definitions of $\Pi_h$ and $\Q_h$
we obtain
\begin{eqnarray*}
|\ell_1(\bu,\Pi_h\bpsi)|&\le&\left| \langle (\nabla \bu-\Q_h\nabla
\bu)\cdot\bn,\;
\Pi_h\bpsi-\{\Pi_h\bpsi\}\rangle_{\pT_h} \right|\\
&\le& \left(\sum_{T\in\T_h}\|\nabla \bu-\Q_h\nabla \bu\|^2_\pT\right)^{1/2}
\left(\sum_{T\in\T_h}\|\Pi_h\bpsi-\{\Pi_h\bpsi\}\|^2_\pT\right)^{1/2}\nonumber \\
&\le& C\left(\sum_{T\in\T_h}h\|\nabla \bu-\Q_h\nabla \bu\|^2_\pT\right)^{1/2}
\left(\sum_{T\in\T_h}h^{-1}\|[\Pi_h\bpsi-\bpsi]\|^2_\pT\right)^{1/2} \nonumber\\
&\le&  Ch^{k+1}|\bu|_{k+1}|\bpsi|_2.\nonumber
\end{eqnarray*}
It follows from (\ref{m30}) and (\ref{errv}) that
\begin{eqnarray*}
|(\nabla_w(\bpsi-\Pi_h\bpsi),\;\nabla_w \bepsilon_h)_{\T_h}|&\le& C\3bar \bepsilon_h\3bar \3bar \bpsi-\Pi_h\bpsi\3bar\\
&\le& Ch^{k+1}|\bu|_{k+1}|\bpsi|_2.
\end{eqnarray*}
The norm equivalence (\ref{happy}) and (\ref{errv}) imply
\begin{eqnarray*}
|(\nabla\bpsi-\Q_h\nabla \bpsi,\ \nabla \bepsilon_h)_{\T_h}|&\le& C(\sum_{T\in\T_h}\|\nabla \bepsilon_h\|_T^2)^{1/2} (\sum_{T\in\T_h}\|\nabla\bpsi-\Q_h\nabla \bpsi\|_T^2)^{1/2}\\
&\le& C(\sum_{T\in\T_h}(\|\nabla(\bu-\Pi_h\bu)\|_T^2+\|\nabla(\Pi_h\bu-\bu_h)\|_T^2))^{1/2}\\
&\times &(\sum_{T\in\T_h}\|\nabla\bpsi-\Q_h\nabla \bpsi\|_T^2)^{1/2}\\
&\le& Ch|\bpsi|_2(h^k|\bu|_{k+1}+\3bar \Pi_h\bu-\bu_h\3bar)\\
&\le& Ch^{k+1}|\bu|_{k+1}|\bpsi|_2.
\end{eqnarray*}
Using (\ref{happy}) and (\ref{errv}), we obtain
\begin{eqnarray*}
|\ell_1(\bpsi,\bepsilon_h)|&=&\left|\l (\nabla \bpsi-\Q_h\nabla \bpsi)\cdot\bn,\
\bepsilon_h-\{\bepsilon_h\}\r_{\partial\T_h}\right|\\
&\le&\sum_{T\in\T_h} h_T^{1/2}\|\nabla \bpsi-\Q_h\nabla \bpsi\|_\pT  h_T^{-1/2}\|[\bepsilon_h]\|_\pT\\
&\le&Ch\|\bpsi\|_2(\sum_{T\in\T_h}h_T^{-1}(\|[\be_h]\|^2_\pT+\|[\bu-\Pi_h\bu]\|^2_\pT)^{1/2}\\
&\le&Ch\|\bpsi\|_2(\3bar\be_h\3bar+(\sum_{T\in\T_h}h_T^{-1}\|[\bu-\Pi_h\bu]\|^2_\pT)^{1/2}\\
&\le&  Ch^{k+1 }|\bu|_{k+1}\|\bpsi\|_2.
\end{eqnarray*}
Combining all the estimates above
with (\ref{mmm22}) yields
$$
\|\bepsilon_h\|^2 \leq C h^{k+1}|\bu|_{k+1} \|\bpsi\|_2.
$$
The estimate (\ref{l2err}) follows from the above inequality and
the regularity assumption (\ref{reg}). We have completed the proof.
\end{proof}

\section{Numerical Experiments}\label{Section:numerical-experiments}

\subsection{Example \ref{e--1}}\label{e--1}
Consider problem  \eqref{moment}--\eqref{bc} with $\Omega=(0,1)^2$.
The source term $\bf$ and the boundary value $\bg$ are chosen so that the exact solution is
\a{
    \bu(x,y)&=\p{-(2-4y) (y-y^2)(x-x^2)^2\\(2-4x) (x-x^2)(y-y^2)^2}, \
   \\  p&= (2-4x) (x-x^2)(2-4y) (y-y^2).
}
In this example, we use uniform triangular grids shown in Figure \ref{grid1}.
In Table \ref{t1}, we list the errors and the orders of convergence.
We can see that the optimal order of convergence is achieved in all finite elements.

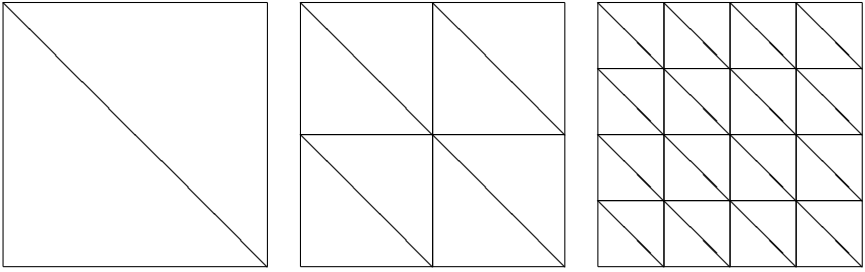
\begin{figure}[h!]
 \begin{center} \setlength\unitlength{1.25pt}
\begin{picture}(260,80)(0,0)
  \def\tr{\begin{picture}(20,20)(0,0)\put(0,0){\line(1,0){20}}\put(0,20){\line(1,0){20}}
          \put(0,0){\line(0,1){20}} \put(20,0){\line(0,1){20}}
   \put(0,20){\line(1,-1){20}}   \end{picture}}
 {\setlength\unitlength{5pt}
 \multiput(0,0)(20,0){1}{\multiput(0,0)(0,20){1}{\tr}}}

  {\setlength\unitlength{2.5pt}
 \multiput(45,0)(20,0){2}{\multiput(0,0)(0,20){2}{\tr}}}

  \multiput(180,0)(20,0){4}{\multiput(0,0)(0,20){4}{\tr}}

 \end{picture}\end{center}
\caption{The first three levels of triangular grids for Examples \ref{e--1} and \ref{e--2}.}
\label{grid1}
\end{figure}

\begin{table}[h!]
  \centering   \renewcommand{\arraystretch}{1.05}
  \caption{Example \ref{e--1}:  Error profiles and convergence rates on grids shown in Figure \ref{grid1}. }
\label{t1}
\begin{tabular}{c|cc|cc|cc}
\hline
level & $\|\bu- \bu_h \|_0 $  &rate &  $\3bar \bu- \bu_h \3bar $ &rate
  &  $\|p - p_h \|_0 $ &rate   \\
\hline
 &\multicolumn{6}{c}{by the $P_1^2$-$P_0$ finite element } \\ \hline
 5&   0.6699E-01 &  1.82&   0.3105E+01 &  1.05&   0.1087E+01 &  1.12 \\
 6&   0.1754E-01 &  1.93&   0.1516E+01 &  1.03&   0.5164E+00 &  1.07 \\
 7&   0.4468E-02 &  1.97&   0.7479E+00 &  1.02&   0.2528E+00 &  1.03 \\
\hline
 &\multicolumn{6}{c}{by the $P_2^2$-$P_1$ finite element } \\ \hline
 5&   0.7769E-03 &  3.10&   0.1607E+00 &  1.95&   0.3320E+00 &  1.68 \\
 6&   0.9439E-04 &  3.04&   0.4059E-01 &  1.98&   0.9258E-01 &  1.84 \\
 7&   0.1170E-04 &  3.01&   0.1018E-01 &  2.00&   0.2443E-01 &  1.92 \\
\hline
 &\multicolumn{6}{c}{by the $P_3^2$-$P_2$ finite element } \\ \hline
 4&   0.3821E-03 &  3.90&   0.4483E-01 &  2.80&   0.1374E+00 &  2.71 \\
 5&   0.2444E-04 &  3.97&   0.5895E-02 &  2.93&   0.1884E-01 &  2.87 \\
 6&   0.1550E-05 &  3.98&   0.7560E-03 &  2.96&   0.2461E-02 &  2.94 \\
 \hline
 &\multicolumn{6}{c}{by the $P_4^2$-$P_3$ finite element } \\ \hline
 4&   0.2350E-04 &  4.85&   0.3140E-02 &  3.81&   0.5804E-02 &  3.85 \\
 5&   0.7645E-06 &  4.94&   0.2082E-03 &  3.91&   0.3818E-03 &  3.93 \\
 6&   0.2433E-07 &  4.97&   0.1341E-04 &  3.96&   0.2447E-04 &  3.96 \\
 \hline
\end{tabular}%
\end{table}%

\subsection{Example \ref{e--2}}\label{e--2}
This example is from Example 1.1 in \cite{John}, for testing the pressure robustness
   of the method.
We solve the following Stokes equations with a different Reynolds number $\mu^{-1}>0$,
\a{
-\mu \Delta\bu+\nabla p&=\p{3(x-x^2)-\frac 12\\0 } \quad
 && \mbox{in}\;\Omega,&&  \\
\nabla\cdot\bu &= 0\quad && \mbox{in}\;\Omega, \\
\bu&= \b 0 \quad&&  \mbox{on}\; \partial\Omega, }where $ \Omega=(0,1)^3$.
The exact solution is, independent of $\mu$,
\an{\label{s2}
     \bu =\p{0\\0}, \quad p&= (x-x^2)(x-\frac 12). }
For this problem,  a non-pressure-robust method, such as the Taylor-Hood method,
   would produce a $\mu$-dependent velocity solution, cf. the numerical results in
     \cite{John} and the data in Tables \ref{t2} and \ref{t3} below.

In this example, we use nonuniform grids shown in Figure \ref{grid1}.
In Table \ref{t2} and Table \ref{t3}, we list the errors and the orders of convergence
  for both the $H(\operatorname{div})$ finite element and the Taylor-Hood element.
We can see that the discrete velocity solution $u_h$ is identically zero for our
   method (only computer round-off error), i.e.,
   of optimal order convergence independently of $\mu$.
But its error is $\mu^{-1}O(h^k)$ for the Taylor-Hood element.
For $P_2$ Taylor-Hood element,  there is one order superconvergence in both $L^2$-norm and
   $H^1$-semi-norm,  noting $\bu=\b 0$ here.
But there is no such superconvergence for $P_k$ $H(\operatorname{div})$ elements, neither
   for $P_3$ Taylor-Hood element.
To show the pollution effect of the Taylor-Hood element,  we plot two solutions in
   Figure \ref{g-u1}.

\begin{table}[h!]
  \centering   \renewcommand{\arraystretch}{1.05}
  \caption{Example \ref{e--2}:  Error profiles and convergence rates on grids shown in Figure \ref{grid1}. }
\label{t2}
\begin{tabular}{c|cc|cc|cc}
\hline
level & $\|\bu- \bu_h \|_0 $  &rate &  $\3bar \bu- \bu_h \3bar $ &rate
  &  $\|p - p_h \|_0 $ &rate   \\
\hline
 &\multicolumn{6}{c}{by the $P_2^2$-$P_1$ Taylor-Hood element, $\mu=1 $ } \\ \hline
 5&   0.2292E-06 &  3.99&   0.2733E-04 &  2.93&   0.9301E-03 &  1.89 \\
 6&   0.1438E-07 &  3.99&   0.3491E-05 &  2.97&   0.2413E-03 &  1.95 \\
 7&   0.9002E-09 &  4.00&   0.4410E-06 &  2.98&   0.6148E-04 &  1.97 \\\hline
 &\multicolumn{6}{c}{by the $P_2^2$-$P_1$ Taylor-Hood element, $\mu=10^{-6}$ } \\ \hline
 5&   0.2292E+00 &  3.99&   0.2733E+02 &  2.93&   0.9301E-03 &  1.89 \\
 6&   0.1438E-01 &  3.99&   0.3491E+01 &  2.97&   0.2413E-03 &  1.95 \\
 7&   0.9002E-03 &  4.00&   0.4410E+00 &  2.98&   0.6148E-04 &  1.97 \\
 \hline
 &\multicolumn{6}{c}{by the $P_2^2$-$P_1$ $H(\operatorname{div})$ finite element, $\mu=1 $ } \\ \hline
 5&   0.1224E-18 &      &   0.1674E-17 &      &   0.5615E-03 &  1.94 \\
 6&   0.1032E-18 &      &   0.1813E-17 &      &   0.1434E-03 &  1.97 \\
 7&   0.1002E-18 &      &   0.1830E-17 &      &   0.3624E-04 &  1.98 \\
\hline
 &\multicolumn{6}{c}{by the $P_2^2$-$P_1$ $H(\operatorname{div})$ finite element, $\mu=10^{-6}$ } \\ \hline
 5&   0.1032E-12 &      &   0.1554E-11 &      &   0.5615E-03 &  1.94 \\
 6&   0.6375E-13 &      &   0.1627E-11 &      &   0.1434E-03 &  1.97 \\
 7&   0.9934E-13 &      &   0.1760E-11 &      &   0.3624E-04 &  1.98 \\
 \hline
\end{tabular}%
\end{table}%

\begin{figure}[htb]\begin{center}
\includegraphics[width=4.4in]{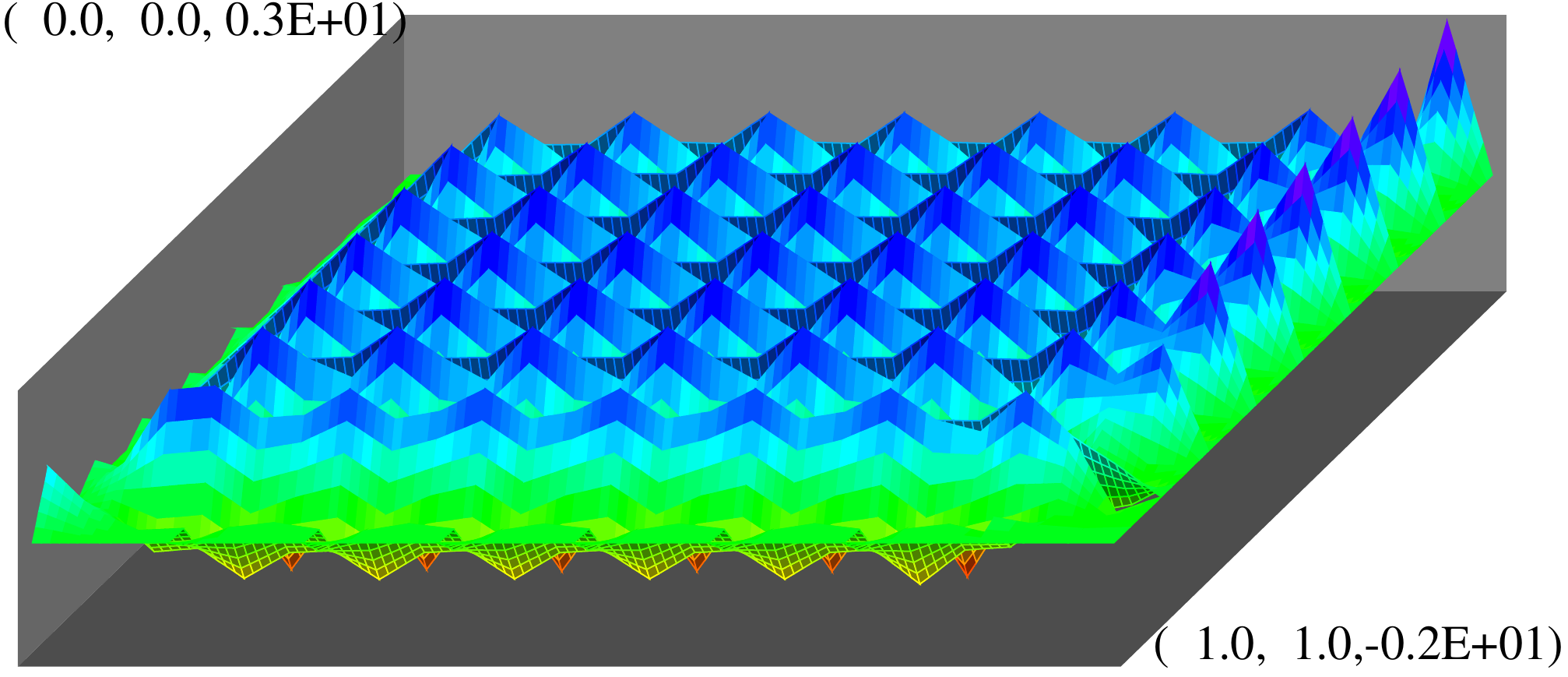} \\
\includegraphics[width=4.4in]{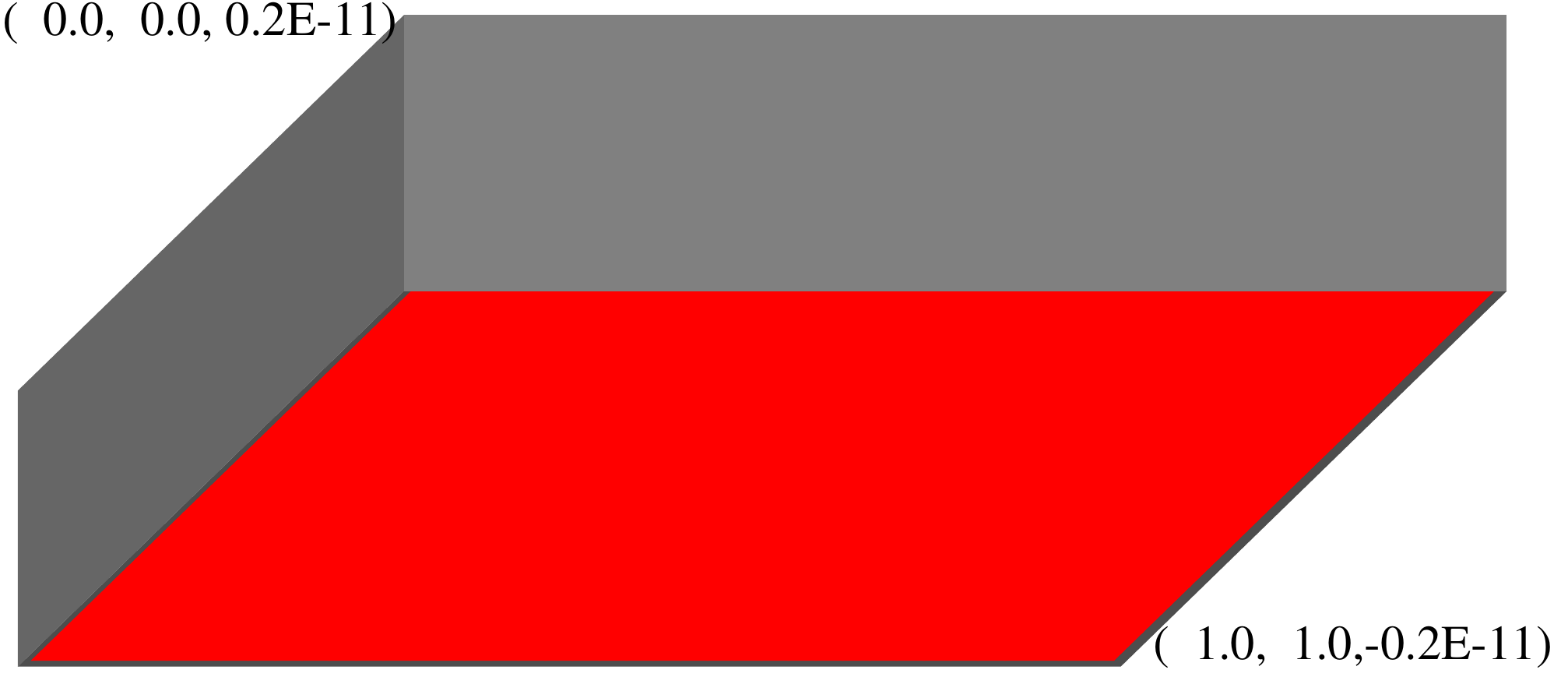}

\caption{The solution $(\bu_h)_1$ of the $P_3$ Taylor-Hood element (top) and
   the $P_3$ $H(\operatorname{div})$ element (bottom), on the 4th level grid,
     for $\mu=10^{-6}$.  }
\label{g-u1}
\end{center}
\end{figure}

\begin{table}[h!]
  \centering   \renewcommand{\arraystretch}{1.05}
  \caption{Example \ref{e--1}:  Error profiles and convergence rates on grids shown in Figure \ref{grid1}. }
\label{t3}
\begin{tabular}{c|cc|cc|cc}
\hline
level & $\|\bu- \bu_h \|_0 $  &rate &  $\3bar \bu- \bu_h \3bar $ &rate
  &  $\|p - p_h \|_0 $ &rate   \\
\hline
 &\multicolumn{6}{c}{by the $P_3^2$-$P_2$ Taylor-Hood element, $\mu=1 $ } \\ \hline
 4&   0.7690E-06 &  3.64&   0.4897E-04 &  2.74&   0.5746E-04 &  3.09 \\
 5&   0.5345E-07 &  3.85&   0.6637E-05 &  2.88&   0.6828E-05 &  3.07 \\
 6&   0.3509E-08 &  3.93&   0.8620E-06 &  2.94&   0.8275E-06 &  3.04 \\
\hline
 &\multicolumn{6}{c}{by the $P_3^2$-$P_2$ Taylor-Hood element, $\mu=10^{-6}$ } \\ \hline
 4&   0.7690E+00 &  3.64&   0.4897E+02 &  2.74&   0.5746E-04 &  3.09 \\
 5&   0.5345E-01 &  3.85&   0.6637E+01 &  2.88&   0.6828E-05 &  3.07 \\
 6&   0.3509E-02 &  3.93&   0.8620E+00 &  2.94&   0.8275E-06 &  3.04 \\
\hline
 &\multicolumn{6}{c}{by the $P_3^2$-$P_2$ $H(\operatorname{div})$ finite element, $\mu=1 $ } \\ \hline
 4&   0.1699E-17 &      &   0.1801E-16 &      &   0.5580E-04 &  3.00 \\
 5&   0.1403E-17 &      &   0.1486E-16 &      &   0.6975E-05 &  3.00 \\
 6&   0.3710E-18 &      &   0.7074E-17 &      &   0.8719E-06 &  3.00 \\
 \hline
 &\multicolumn{6}{c}{by the $P_3^2$-$P_2$ $H(\operatorname{div})$ finite element, $\mu=10^{-6}$ } \\ \hline
 4&   0.1655E-11 &      &   0.1789E-10 &      &   0.5580E-04 &  3.00 \\
 5&   0.1318E-11 &      &   0.1449E-10 &      &   0.6975E-05 &  3.00 \\
 6&   0.3331E-12 &      &   0.6899E-11 &      &   0.8719E-06 &  3.00 \\
 \hline
\end{tabular}%
\end{table}%

\subsection{Example \ref{e--3}}\label{e--3}
Consider problem \eqref{moment} with $\Omega=(0,1)^3$.
The source term $f$ and the boundary value $g$ are chosen so that the exact solution is
\begin{equation*}
    \bu(x,y,z)=\p{y^2\\z^2\\x^2}, \quad p(x,y,z)=yz-\frac 14.
\end{equation*}
We use uniform tetrahedral meshes shown in Figure \ref{grid3}.
The results of the $P_2^2$-$P_{1}$ $H(\operatorname{div})$ mixed finite element method
    are listed in Table \ref{t4}.
The method converges at the optimal order in the usual norms.

\begin{figure}[h!]
\begin{center}
 \setlength\unitlength{1pt}
    \begin{picture}(320,118)(0,3)
    \put(0,0){\begin{picture}(110,110)(0,0)
       \multiput(0,0)(80,0){2}{\line(0,1){80}}  \multiput(0,0)(0,80){2}{\line(1,0){80}}
       \multiput(0,80)(80,0){2}{\line(1,1){20}} \multiput(0,80)(20,20){2}{\line(1,0){80}}
       \multiput(80,0)(0,80){2}{\line(1,1){20}}  \multiput(80,0)(20,20){2}{\line(0,1){80}}
    \put(80,0){\line(-1,1){80}}\put(80,0){\line(1,5){20}}\put(80,80){\line(-3,1){60}}
      \end{picture}}
    \put(110,0){\begin{picture}(110,110)(0,0)
       \multiput(0,0)(40,0){3}{\line(0,1){80}}  \multiput(0,0)(0,40){3}{\line(1,0){80}}
       \multiput(0,80)(40,0){3}{\line(1,1){20}} \multiput(0,80)(10,10){3}{\line(1,0){80}}
       \multiput(80,0)(0,40){3}{\line(1,1){20}}  \multiput(80,0)(10,10){3}{\line(0,1){80}}
    \put(80,0){\line(-1,1){80}}\put(80,0){\line(1,5){20}}\put(80,80){\line(-3,1){60}}
       \multiput(40,0)(40,40){2}{\line(-1,1){40}}
        \multiput(80,40)(10,-30){2}{\line(1,5){10}}
        \multiput(40,80)(50,10){2}{\line(-3,1){30}}
      \end{picture}}
    \put(220,0){\begin{picture}(110,110)(0,0)
       \multiput(0,0)(20,0){5}{\line(0,1){80}}  \multiput(0,0)(0,20){5}{\line(1,0){80}}
       \multiput(0,80)(20,0){5}{\line(1,1){20}} \multiput(0,80)(5,5){5}{\line(1,0){80}}
       \multiput(80,0)(0,20){5}{\line(1,1){20}}  \multiput(80,0)(5,5){5}{\line(0,1){80}}
    \put(80,0){\line(-1,1){80}}\put(80,0){\line(1,5){20}}\put(80,80){\line(-3,1){60}}
       \multiput(40,0)(40,40){2}{\line(-1,1){40}}
        \multiput(80,40)(10,-30){2}{\line(1,5){10}}
        \multiput(40,80)(50,10){2}{\line(-3,1){30}}

       \multiput(20,0)(60,60){2}{\line(-1,1){20}}   \multiput(60,0)(20,20){2}{\line(-1,1){60}}
        \multiput(80,60)(15,-45){2}{\line(1,5){5}} \multiput(80,20)(5,-15){2}{\line(1,5){15}}
        \multiput(20,80)(75,15){2}{\line(-3,1){15}}\multiput(60,80)(25,5){2}{\line(-3,1){45}}
      \end{picture}}

    \end{picture}
    \end{center}
\caption{  The first three levels of grids used in Example \ref{e--3}. }
\label{grid3}
\end{figure}
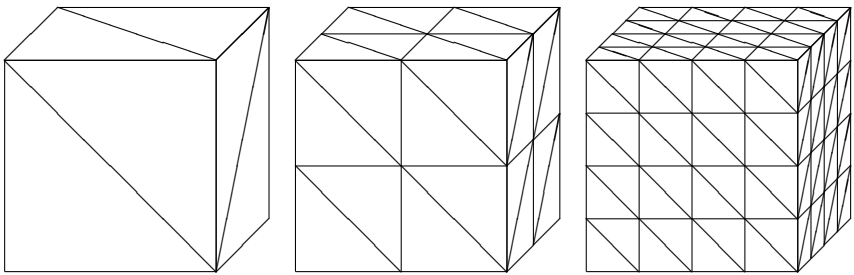

\begin{table}[h!]
  \centering \renewcommand{\arraystretch}{1.1}
  \caption{Example \ref{e--3}:  Error profiles and convergence rates on 3D grids shown in Figure \ref{grid3}. }
\label{t4}
\begin{tabular}{c|cc|cc|cc}
\hline
Grid & $\3bar \bu- \bu_h \3bar $ &rate & $\|\bu- \bu_h \|_0 $  &rate &
    $\|p - p_h \|_0 $ &rate   \\
\hline
 1&    0.2599E+00& 0.00&    0.1654E-01& 0.00&    0.6503E+00& 0.00 \\
 2&    0.2155E+00& 0.27&    0.1091E-01& 0.60&    0.3843E+00& 0.76 \\
 3&    0.1212E+00& 0.83&    0.3520E-02& 1.63&    0.1232E+00& 1.64 \\
 4&    0.6209E-01& 0.97&    0.1006E-02& 1.81&    0.3879E-01& 1.67 \\
 \hline
\end{tabular}%
\end{table}%

\end{document}